\documentclass[12pt,a4paper]{amsart}
\usepackage{graphicx,amssymb}
\usepackage{amsmath,amssymb,amscd}
\input xy
\xyoption{all}
\usepackage[all]{xy}
\usepackage{tikz}

\newcommand{\lra}{\longrightarrow}

\newcommand{\cU}{\mathcal{U}}
\newcommand{\cO}{\mathcal{O}}

\newcommand{\Pic}{\mbox{Pic}}

\newcommand{\gr}{\text{gr}}

\newcommand{\im}{\operatorname{Im}}
\newcommand{\um}{\"}
\theoremstyle{plain}
\newtheorem{theorem}{Theorem}[section]
\newtheorem{lem}[theorem]{Lemma}
\newtheorem{prop}[theorem]{Proposition}
\newtheorem{cor}[theorem]{Corollary}

\newtheorem{rem}[theorem]{Remark}
\newtheorem{ex}[theorem]{Example}

\newtheorem{df}[theorem]{Definition}
\numberwithin{equation}{section}
\begin{document}
\title[Twisted BN loci]{New examples of twisted Brill-Noether loci I}

\author{L. Brambila-Paz}
\address{CIMAT\\Apdo. Postal 402\\ C.P. 36240\\ Guanajuato\\ Mexico}   
\email{lebp@cimat.mx}
\author{P. E. Newstead}
\address{Department of Mathematical Sciences\\
             University of Liverpool\\
              Peach Street, Liverpool L69 7ZL, UK}
\email{newstead@liv.ac.uk}  
 
\date{\today}

\thanks{The authors are members of the research group VBAC (Vector Bundles on Algebraic Curves).The first author acknowledges the support of CONACYT proj. 251938.}

\maketitle
\begin{center}{\it Dedicated to Oscar Garc\'{\i}a-Prada on the occasion of his 60th birthday}

\end{center}

\begin{abstract}
Our purpose in this paper is to construct new examples of twisted Brill-Noether loci on curves of genus $g\ge2$. Many of these examples have negative expected dimension. We deduce also the existence of a new region in the Brill-Noether map, whose points support non-empty standard Brill-Noether loci.

\end{abstract}
\tableofcontents

\section{Introduction}\label{intro}
Our object in this paper and its successor \cite{bpn} is to construct twisted Brill-Noether loci different from those constructed in \cite{hhn}. In particular, the relevance of many of these examples is that they have negative Brill-Noether numbers, whereas the examples constructed in \cite[Theorem 1.1]{hhn} all have positive Brill-Noether numbers. We deduce also the existence of a new region in the Brill-Noether map (see below for the definition), the points of which support non-empty Brill-Noether loci, whose existence was not previously known. This implies that our new examples also differ from those constructed in \cite[Section 9]{hhn} (see Remark \ref{r26}(iii)). In this paper, we consider applications of the results of \cite{bgn,bmno,m1,m3,m2,te} together with Butler's Theorem \cite[Theorem 2.1]{bu1}. In \cite{bpn}, we will, in particular, use known cases of Butler's Conjecture \cite[Conjecture 2]{bu2} (see \cite{bbn1,bbn2,bmgno}) to obtain further examples. Our methods will continue to apply as further cases of Butler's Conjecture are proved. In what follows, we shall usually abbreviate Brill-Noether to BN.

Let $C$ be a smooth irreducible projective curve of genus $g\ge2$ defined over the complex numbers. For $i=1,2$, let $M_i:=M(n_i,d_i)$ denote the moduli space of stable bundles of rank $n_i$ and degree $d_i$. Similarly, we write $\widetilde{M}_i$ for the corresponding moduli space of S-equivalence classes of semistable bundles. 

Let $E_2$ be any vector bundle on $C$ of rank $n_2$ and degree $d_2$. For any integers $n_1$, $d_1$, $k$ with $n_1\ge1$, one can define the \textit{twisted BN locus}
 \begin{equation}\label{eq4}
B(n_1,d_1,k)(E_2):=\{E_1\in M_1\,|\,h^0(E_1\otimes E_2)\ge k\}.
\end{equation}
The locus $B(n_1,d_1,k)(E_2)$ is a closed subscheme of $M_1$. In fact, it is a determinantal locus, being the pullback of a universal determinantal variety of codimension $k(k-\chi)$,
where
\begin{equation}\label{eq24}
\chi:=n_2d_1+n_1d_2-n_1n_2(g-1).
\end{equation}
It follows that, if $\beta(n_1,d_1,k)(E_2)\le n_1^2(g-1)+1$, then every irreducible component of $B(n_1,d_1,k)(E_2)$ has dimension greater than or equal to the \textit{BN number}
\begin{equation}\label{eq5}
\beta(n_1,d_1,k)(E_2):=n_1^2(g-1)+1-k(k-\chi),
\end{equation}
This number is often referred to as the \textit{expected dimension}.

We define also
\[\widetilde{B}(n_1,d_1,k)(E_2):=\{[E_1]\in \widetilde{M}_1\,|\,h^0(\operatorname{gr}E_1 \otimes E_2)\ge k\},\]
where $[E_1]$ denotes the S-equivalence class of $E_1$ and $\gr E_1$ is the corresponding graded bundle.
This is a closed subscheme of $\widetilde{M}_1$, but can have components of dimension less than the BN number.

 In the special case $E_2=\cO_C$, we obtain the standard (untwisted) higher rank BN locus, denoted by $B(n_1,d_1,k)$, with expected dimension 
\begin{equation}\label{eq44}
\beta(n_1,d_1,k):=n_1^2(g-1)+1-k(k-d_1+n_1(g-1)).
\end{equation} In particular, $B(1,d_1,k)$ is the classical BN locus $W^{k-1}_{d_1}$.

In the case $n_1=1$, twisted BN loci have been studied for some time starting with \cite{gh,laz1} (see \cite[Theorem 2.1]{hhn} for full details and further references). These loci include the special case of maximal line subbundles. A systematic study for arbitrary $n_1$ was started in \cite{hhn}, where, in particular, examples of twisted BN loci were constructed \cite[Theorem 1.1]{hhn} using the technique of \cite{m3}. As we shall see in Theorem \ref{t1}, these have the important feature that the BN number \eqref{eq5} is positive. Many of them (maybe, all) have the expected dimension \cite[Theorem 1.2]{hhn}.

We can now introduce the loci which we are going to study in this paper. This is a special case of a more general construction introduced in \cite{hhn}, where it was used primarily to obtain results for $B(n_1,d_1,k)(E_2)$. If $\gcd(n_i,d_i)=1$ for $i=1,2$, then there exist universal bundles $\cU_i$ on $M_i\times C$ and, for any integer $k$, we define the \textit{universal twisted BN locus} as follows:
\begin{equation}\label{eq1}
B^k(\cU_1,\cU_2):=\{(E_1,E_2)\in M_1\times M_2\, |\,h^0(E_1\otimes E_2)\ge k\}.
\end{equation}
In the non-coprime case, we need to lift to \'etale coverings of $M_1$ and $M_2$ in order to construct ``universal'' bundles \cite[Proposition 2.4]{nr}. However \eqref{eq1} still makes sense and we shall use the same notation. The locus $B^k(\cU_1,\cU_2)$ is a closed subscheme of $M_1\times M_2$; if $k\le0$, then $B^k(\cU_1,\cU_2)=M_1\times M_2$. The structure of $B^k(\cU_1,\cU_2)$ as a determinantal locus allows us to define a BN number (see \cite[(3.1)]{hhn}
\begin{equation}\label{eq2}
\beta^k(\cU_1,\cU_2):=\dim M_1+\dim M_2-k(k-\chi),
\end{equation}
which may again be referred to as the \textit{expected dimension}. Note that  $\chi=\chi(E_1\otimes E_2)$ for all $(E_1,E_2)\in M_1\times M_2$.  If $\beta^k(\cU_1,\cU_2)\le \dim M_1+\dim M_2$, then every irreducible component of $B^k(\cU_1,\cU_2)$ has dimension  greater than or equal to $\beta^k(\cU_1,\cU_2)$. {We define similarly the corresponding semistable locus 
\[\widetilde{B}(\cU_1,\cU_2):=\{([E_1],[E_2])\in \widetilde{M}_1\times\widetilde{M}_2\,|\,h^0(\gr E_1\otimes\gr E_2)\ge k\}.\]

Note that the definition is symmetric and $B^k(\cU_1,\cU_2)\cong B^k(\cU_2,\cU_1)$. Moreover, by Serre duality, $B^k(\cU_1,\cU_2)\cong B^{k_1}(\cU_1^*,\cU_2^*\otimes p_C^*K_C)$, where $k_1=k-\chi$ and $K_C$ denotes the canonical line bundle on $C$. Note also that $B^k(\cU_1,\cU_2)\cong B^k(\cU_1\otimes p_C^*L^*,\cU_2\otimes p_C^*L)$ for any line bundle $L$. We take account of these possibilities in Theorem \ref{t1}(ii) and Remark  \ref{r3}(iii) (see also Remark \ref{r4}(i)).

\begin{rem}\label{r16}\begin{em}The following conditions are equivalent:
\begin{itemize}
\item $B^k(\cU_1,\cU_2)\ne\emptyset$;
\item $B(n_1,d_1,k)(E_2)\ne\emptyset$ for some $E_2\in M_2$;
\item $B(n_2,d_2,k)(E_1)\ne\emptyset$ for some $E_1\in M_1$.
\end{itemize}
\end{em}\end{rem}

Many of our results are valid for any smooth curve. For some, however, we require the concept of a Petri curve. The curve $C$ is said to be \textit{Petri} if the classical Petri map
$H^0(L)\otimes H^0(K_C\otimes L^*)\lra H^0(K_C)$
is injective for every line bundle $L$ on $C$. The general curve is Petri in the sense that Petri curves of genus $g$ form a non-empty open subset of the moduli space of curves of genus $g$ \cite{g,laz}. By classical BN theory, if $C$ is Petri, the BN locus $B(1,d_1,k)$ has dimension precisely $\beta(1,d_1,k)$ whenever $0\le\beta(1,d_1,k)\le g$ and is empty when $\beta(1,d_1,k)<0$. Sometimes we will not be able to obtain results for arbitrary Petri curves but only for curves which are \textit{general} in the sense that they live in an unspecified open subset of the moduli space of curves of genus $g$.

The BN loci $B(1,d_1,k)(E_2)$, with $E_2$ general, behave in a similar way to the classical BN loci \cite[Theorem 2.1]{hhn}. By Remark \ref{r16}, this implies non-emptiness results for
$B^k(\cU_1,\cU_2)$ when $n_1=1$ or $n_2=1$. Moreover, when $n_2=1$, we have an isomorphism
\[B^k(\cU_1,\cU_2)\lra B(n_1,d_1+n_1d_2,k)\times \Pic^{d_2}(C):\ (E_1,L)\mapsto(E_1\otimes L,L),\]
giving a description of $B^k(\cU_1,\cU_2)$ in terms of an untwisted BN locus. We know many cases in which such loci are non-empty; see, in particular, Section \ref{back}. In  this paper, we shall assume that $n_i\ge2$ for $i=1,2$.

We can now state our first main theorem.

\begin{theorem}\label{t7}
Suppose that $C$ is a smooth curve of genus $g\ge2$ and that $d_i$, $k_i$ $(i=1,2)$ are positive integers. Let $B(n_i,d_i,k_i)$ be non-empty BN loci with $n_i\ge2$ and suppose that $k=k_1k_2$. 
\begin{itemize}
\item[(i)] If $d_1<2n_1$ and $d_2\le2gn_2$, then $B^k(\cU_1,\cU_2)\ne\emptyset$.
\item[(ii)] If, in addition, $d_i/n_i=\mu_i$, $k_i/n_i=\lambda_i$ and $\mu_1+\mu_2<\lambda_1\lambda_2+g-1$, then, for  fixed $\mu_i$, $\lambda_i$ and sufficiently large $n_1$, $n_2$, $\beta^k(\cU_1,\cU_2)<0$.
\end{itemize}
\end{theorem}

Part (i) of this theorem will be stated and proved in a more detailed version as Theorem \ref{t6}; (ii) is Lemma \ref{l10}. Using the results of Section \ref{back}, it is easy to find many cases where the hypotheses of the theorem are satisfied, thus obtaining the examples we are seeking (see Example \ref{ex2}). There exist examples even when $n_1=n_2=2$ (see Example \ref {ex1}). 

Our second main theorem (Theorem \ref{t8}) is an important application to BN theory which shows that this construction gives a new region in the BN map (for the definition, see Subsection 2.1).}

Returning to twisted BN loci, our third main theorem is of a rather different nature, in that $d_2$ is negative and so $h^0(E_2)=0$ for all $E_2\in M_2$.

\begin{theorem}\label{t5} Suppose that  $C$ is a smooth curve of genus $g\ge 2$, $n_1\ge2$, $k_1>n_1$, $n$ a positive integer and either $d>2ng$ or $d=2ng$ and $C$ is non-hyperelliptic. Suppose further that $B(n_1,d_1,k_1)\ne\emptyset$ and that 
\[k\le (d-n(g-1))(k_1-n_1)-nd_1.\] Let $(n_2,d_2)=(d-ng,-d)$. Then
$B^k(\cU_1,\cU_2)\ne\emptyset$.
If, in addition, $k=d(k_1-n_1)-e$, where $e\ge n(g-1)(k_1-n_1)+nd_1$, and
\begin{equation*}
 d_1<k_1+n_1(g-1)-\frac{g-1}{k_1-n_1},
\end{equation*}
then,  for any fixed values of $n_1$, $d_1$, $k_1$, $n$ and $e$, and $d\gg0$,
$\beta^k(\cU_1,\cU_2)<0$.
\end{theorem}

A more detailed version of this theorem will be stated and proved as Theorem \ref{t4}. In Corollary \ref{c6} and the related Example \ref{ex5}, we show that the requirement that $B(n_1,d_1,k_1)\ne\emptyset$ is compatible with the condition $d_1<k_1+n_1(g-1)-\frac{g-1}{k_1-n_1}$. Hence Theorem \ref{t4} does lead to examples of twisted BN loci with negative expected dimension.

Section \ref{back} is concerned with the non-emptiness of $B(n_1,d_1,k_1)$. In subsection \ref{ss22}, we introduce the BN map and describe some sufficient conditions for non-emptiness.  Subsection \ref{ss25} contains a description of kernel bundles and the dual span construction. We include also a description of BN loci for bundles of slope $\le2$ (subsection \ref{ss24}). In section \ref{twisted}, we obtain an extended version of the main existence theorem of \cite{hhn}. Sections \ref{first} and \ref{kernel} contain our main results. 

The dual span construction works also when we consider pairs $(E,V)$ where $V$ is a subspace of $H^0(E)$ which generates $E$. In the case where $E$ is a line bundle, there are then further stability results, which have consequences for the non-emptiness of $B(n_1,d_1,n_1+1)$. We shall investigate applications of this to twisted BN loci in a future paper \cite{bpn}.

Throughout the paper, $C$ will denote a smooth irreducible projective curve of genus $g\ge2$ defined over ${\mathbb C}$. For a vector bundle $E$ on $C$ of rank $n$ and degree $d$, we denote by $\mu(E):=\frac{d}n$ the slope of $E$. We write also $h^0(E)$ for the dimension of the space of sections $H^0(E)$. The canonical line bundle on $C$ is denoted by $K_C$. For any real number $a$, we write $\lfloor a\rfloor$ and $\lceil a\rceil$ for the largest integer $\le a$ and the smallest integer $\ge a$, respectively.

We thank Lilia Alanis Lopez for help with the figures and the referee for some useful comments.

\section{Non-emptiness of BN loci}\label{back}
In this section, we will recall and extend some of the principal known results on the non-emptiness of the BN loci $B(n_1,d_1,k_1)$. The answer is completely known for $k_1\le n_1$, but the problem is much more complicated (and still not completely solved) for $k_1>n_1$ (except for $g\le3$ and (almost) for hyperelliptic curves) (see \cite{bmno}). For Petri curves, the na\um{\i}ve conjecture is that $B(n_1,d_1,k_1)$ is non-empty if and only if $\beta(n_1,d_1,k_1)\ge0$. From \eqref{eq44}, this is equivalent to
\begin{equation}\label{eq45}
d_1\ge k_1+n_1(g-1)-\frac{n_1(g-1)}{k_1/n_1}-\frac1{k_1}.
\end{equation}
In fact, numerous examples show that \eqref{eq45} is neither necessary nor sufficient for non-emptiness of $B(n_1,d_1,k_1)$ (see, for example,\cite{bgn}, \cite{bf}, \cite{m1}, \cite[Remark 5.5]{ln5}, \cite{bh}).

\subsection{The BN map}\label{ss22}

Following \cite{bmno}, the non-emptiness of $B(n_1,d_1,k_1)$ can be understood in terms of the \textit{BN map}. For this, we map any non-empty BN locus $B(n_1,d_1,k_1)$ (or $\widetilde{B}(n_1,d_1,k_1)$) with $0\le d_1\le n_1(2g-2)$ to the point in the plane 
\begin{equation}\label{eq54}
(\mu,\lambda)=(d_1/n_1,k_1/n_1).
\end{equation}
We shall say that $(\mu,\lambda)$ \textit{supports} $B(n_1,d_1,k_1)$ (resp., $\widetilde{B}(n_1,d_1,k_1)$) if \eqref{eq54} holds and $B(n_1,d_1,k_1)\ne\emptyset$ (resp., $\widetilde{B}(n_1,d_1,k_1)\ne\emptyset$). We are interested here in regions in the $(\mu,\lambda)$-plane for which it is known that all points with rational coordinates support infinitely many BN loci. 

The first examples of such regions were determined by Teixidor i Bigas \cite{te} and Mercat \cite{m3}.
To describe these in terms of the BN map, we follow \cite{bmno}. For $s\ge1$, write $\widehat{\eta}'(s):=s+g-2-\left\lfloor\frac{g-1}s\right\rfloor$; note that $d\ge\widehat{\eta}'(s)$ if and only if $\beta(1,d+1,s)\ge1$, and that $\widehat{\eta}'(1)=0$, $\widehat{\eta}'(g)=2g-2$. Now define a function $t_g$ for $0\le\mu\le2g-2$ as follows:
\begin{equation}\label{eq65}
t_g(\mu):=\begin{cases}0,&\mu=0;\\
\mu-\lceil\mu\rceil+s,&\mu\in]\widehat{\eta}'(s),\widehat{\eta}'(s)+1];\\
s,&\mu\in]\widehat{\eta}'(s)+1,\widehat{\eta}'(s+1)].
\end{cases}\end{equation}
We define T to be the region in the $(\mu,\lambda)$-plane given by
\begin{equation}\label{eq60}
\mu,\lambda\in{\mathbb Q},\ 0\le\mu\le2g-2,\ 0<\lambda\le t_g(\mu)
\end{equation}
(see Figure 1). Note that Clifford's Theorem can be stated in the form 
\[\lambda\le\frac{\mu}2+1\]
and the boundary line for this inequality (the \textit{Clifford line}) is shown in the figure. Moreover, the condition $\beta(n_1,d_1,k_1)=1$ becomes
\[\lambda(\lambda-\mu+g-1)=g-1\]
(see \eqref{eq44}), which defines a branch of a hyperbola with asymptotes $\lambda=0$ and $\lambda=\mu-g+1$ (the \textit{BN curve}), and is also shown in the figure.

The following lemma summarises the results of Teixidor \cite{te} and Mercat \cite{m3}. Teixidor's proof is for general $C$; for $\widetilde{B}(n_1,d_1,k_1)$, this is sufficient to prove it for any smooth curve. Mercat's result \cite[Th\'eor\`eme A-5]{m3} is valid for any smooth curve and implies our lemma. For the version stated below, see \cite[Section 5]{bmno}.
\begin{lem}\label{l2}
Let $C$ be a smooth curve of genus $g\ge2$ and suppose that $n_1\ge2$. If $(\mu,\lambda)\in \operatorname{T}$, then $(\mu,\lambda)$ supports $\widetilde{B}(n_1,d_1,k_1)$ for all values of $n_1$ for which $n_1\mu$ and $n_1\lambda$ are both integers. The same is true for $B(n_1,d_1,k_1)$, provided that the points
\begin{equation}\label{eq67}
(\mu,\lambda)=(\widehat{\eta}'(s)+1,\lambda),\ s-1<\lambda\le s,\ \widehat{\eta}'(s)+1\ne\widehat{\eta}'(s+1).
\end{equation}
are excluded.
\end{lem}

\begin{rem}\begin{em}\label{r21}
Lemma \ref{l2} is equivalent to the statement that $\widetilde{B}(n_1,d_1,k_1)\ne\emptyset$ if  \begin{equation}\label{eq40}
d_1\ge 
k_1+n_1(g-1)-n_1\left\lfloor\frac{g-1}{\lceil k_1/n_1\rceil}\right\rfloor.
\end{equation}
This can also be stated as
\begin{equation}\label{eq26}
\mu\ge \lambda+g-1-\left\lfloor\frac{g-1}{\left\lceil \lambda\right\rceil}\right\rfloor.
\end{equation}
For $B(n_1,d_1,k_1)$, the case
\begin{equation}\label{eq59}d_1= n_1\left(\left\lceil\frac{k_1}{n_1}\right\rceil+g-1- \left\lfloor\frac{g-1}{\lceil k_1/n_1\rceil}\right\rfloor\right)
\end{equation}
must be excluded when $\beta(1,d_1/n_1+1,\left\lceil k_1/n_1\right\rceil+1)\le0$. \eqref{eq40} and \eqref{eq59} are essentially a restatement of Teixidor's conditions \cite{te}  and relate closely to Mercat's construction \cite{m3}.
\end{em}\end{rem}

Another method of constructing non-empty BN loci on a smooth curve $C$ is described in \cite{bmno}. This starts from the results of \cite{bgn,m1,m2}, which give precise conditions for non-emptiness when $d_1\le2n_1$, and proceeds by tensoring by line bundles $L$ with $h^0(L)=s\ge1$. Serre duality extends the region further and we denote the result by BMNO. To describe BMNO in the BN map, we write
\[\hat{\eta}(s):=s+g-1-\left\lfloor\frac{g}s\right\rfloor\]
for $s\ge1$.
The condition $d\ge\hat{\eta}(s)$ is then equivalent to $\beta(1,d,s)\ge0$. On a Petri curve, this is equivalent by classical BN theory to the assertion that $B(1,d,s)\ne\emptyset$. Note that $\hat{\eta}(1)=0$ and, on a Petri curve, $\hat{\eta}(2)$ is the gonality of $C$. We now define a function $f_g(\mu)$ for $0\le\mu\le g-1$  as follows:
\begin{equation}\label{eq27}
f_g(\mu):=\begin{cases}\frac{s}g(\mu-\left\lceil\mu\right\rceil)+s,&\mu\in[\hat{\eta}(s),\hat{\eta}(s)+1];\\
\frac{s}g(\mu-\left\lceil\mu\right\rceil+1)+s,&\mu\in]\hat{\eta}(s)+1,\hat{\eta}(s+1)-1];\\
\frac{\hat{\eta}(s+1)-s}g(\mu-\left\lceil\mu\right\rceil+1)+s,&\mu\in]\hat{\eta}(s+1)-1,\hat{\eta}(s+1)[.
\end{cases}
\end{equation}
We extend the definition of $f_g$, which is modified from that in \cite{bmno} to allow $\mu=0$, to the whole interval $[0,2g-2]$ by Serre duality. The region BMNO in the $(\mu,\lambda)$-plane is then defined by the conditions 
\[0\le\mu\le2g-2,\ 0<\lambda\le f_g(\mu).\]
 For $g=10$, this region is illustrated in Figure 1, which also includes the top boundary of T, the Clifford line, the BN curve and an indication of the excluded points from Lemmas \ref{l2} and \ref{l5} and Remark \ref{r12}(ii). This is a copy of \cite[Figure 6]{bmno} with some added information.

The following lemma summarises the principal result of \cite{bmno}, taking account of the results of \cite{m2} (see \cite[Remark 4.4]{bmno}).

\begin{lem}\label{l5}Let $C$ be a smooth curve of genus $g\ge2$ and $n_1\ge2$. If $(\mu,\lambda)\in\operatorname{BMNO}$, then $(\mu,\lambda)$ supports $\widetilde{B}(n_1,d_1,k_1)$ for infinitely many values of $n_1$. If $C$ is non-hyperelliptic, the same holds for $B(n_1,d_1,k_1)$, provided that points of the form $(\mu,\lambda)$ such that $\mu=\hat{\eta}(s)$ with $\mu\le g-1$ and $\lambda>\frac{s-1}g+s-1$, together with their Serre duals and the points $((\widehat{\eta}(s)+1,s)$, are excluded.
\end{lem}

\begin{center}
\includegraphics[width=13.0cm]{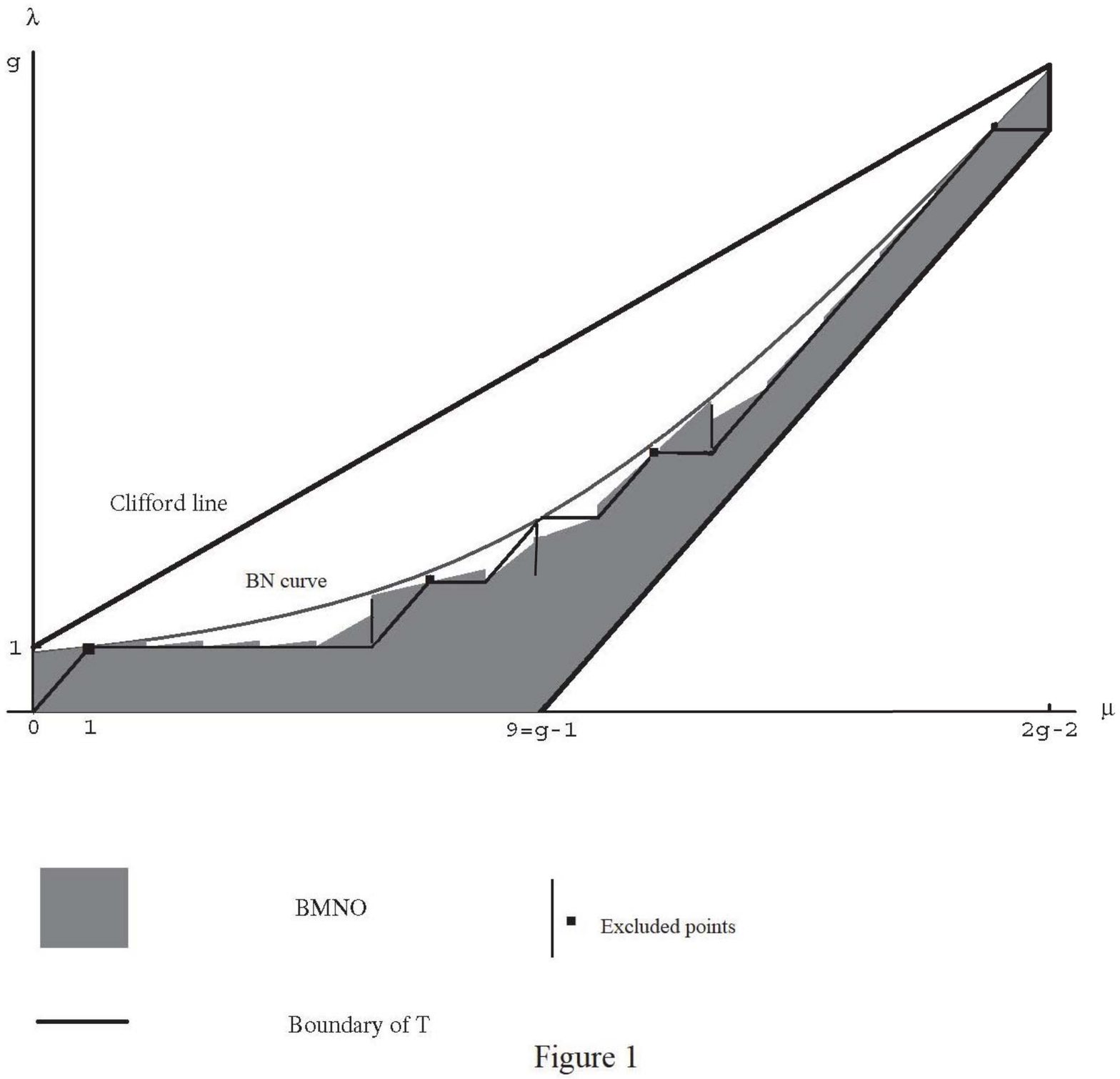}
\end{center}

\begin{rem}\label{r12}\begin{em}
(i) \cite[Remarks 4.2 and 4.3]{bmno} provide information about which values of $n_1$ are included in the above result for any given $(\mu,\lambda)$. In particular, there is a dense set of points in BMNO for which all values of $n_1$ with $n_1\mu, n_1\lambda\in{\mathbb Z}$ are allowed.

(ii) Some points excluded in Lemma \ref{l2} are included in Lemma \ref{l5}. In fact, if $g$ is a multiple of $s$, then $\widehat{\eta}'(s)=\widehat{\eta}(s)$, and  the points $(\widehat{\eta}(s)+1,\lambda)$ for $s-1<\lambda<s$ are excluded by \eqref{eq67}, but included by Lemma \ref{l5}.

(iii) For $g\ge4$, there are many non-empty $B(n_1,d_1,k_1)$ corresponding to points outside the union of T and BMNO; see, for example, the figures in \cite{ln1,ln2,ln5}. We shall construct later a new region of the BN map, giving new points for all $g\ge5$ (see Theorem \ref{t8}).

(iv)\cite{bmno} also contains comprehensive results for hyperelliptic curves.
\end{em}\end{rem}

\subsection{Kernel and dual span bundles}\label{ss25}
Let $E$ be a generated bundle of rank $n$ and degree $d$. Consider the exact sequence 
\begin{equation}\label{eq6}
0\lra D_E^*\lra H^0(E)\otimes\cO\lra E\lra0.
\end{equation}
The bundle $D_E^*$ may be referred to as a \textit{kernel bundle} (or \textit{syzygy bundle}) and has rank $h^0(E)-ng$ and degree $-d$. Dualising \eqref{eq6}, we have
\begin{equation}\label{eq58}
0\lra E^*\lra H^0(E)^*\otimes\cO\lra D_E\lra0.
\end{equation}
Following \cite{bu2}, we call $D_E$ a \textit{dual span bundle}.

The following theorem addresses the stability of $D_E$ and summarises results of Butler and Mercat (see also \cite[Theorem 3.11]{bmgno}). We write
 \begin{equation}\label{eq43}
 S:=\{E\in M(n,d)|D_E\text{ is stable}\}.
 \end{equation}

\begin{theorem}\label{l1} Let $C$ be a smooth curve of genus $g\ge2$. Suppose that either $d>2ng$ and $E\in M(n,d)$ or $d=2ng$, $C$ is non-hyperelliptic and $E$ is a general element of $M(n,d)$. Then $E$ is generated and $D_E$ is stable of slope $\le2$. In particular, $S$ is a non-empty open subset of $M(n,d)$ and $S=M(n,d)$ when $d>2ng$. Moreover, the morphism 
\begin{equation}\label{eq21}
S\lra B(d-ng,d,d-n(g-1)): E\mapsto D_E
\end{equation}
 is an isomorphism.
\end{theorem}
\begin{proof} Any stable bundle of slope $>2g-1$ is generated with $h^1(E)=0$. Hence $h^0(E)=d-n(g-1)$. It follows from \eqref{eq58} that $h^0(D_E)\ge d-n(g-1)$. Moreover, under the given hypotheses, $D_E$ is stable by \cite[Theorem 2.1]{bu1}. It is easy to check that $\mu(D_E)\le2$.
The fact that \eqref{eq21} is an isomorphism follows from \cite[Th\'eor\`eme 2-B-1]{m1} and \cite[Proposition 2]{m2}.
\end{proof}

\subsection{Small slope}\label{ss24}
When $d_1<2n_1$, Lemma \ref{l5} includes a complete answer to the question of the non-emptiness of $B(n_1,d_1,k_1)$. The references \cite{bgn,m1} (see also \cite{top3}) contain much further information about these BN loci, in particular that there exists $E_1\in B(n_1,d_1,k_1)$, taking one of a number of special forms.

\begin{theorem}\label{t9} Let $C$ be a curve of genus $g\ge2$. If $d_1<2n_1$ and $B(n_1,d_1,k_1)\ne\emptyset$, then there exists $E_1\in B(n_1,d_1,k_1)$ taking one of the following forms.
\begin{itemize}
\item[I]  $d_1<n_1+g$. $E_1$ fits into an exact sequence 
\begin{equation}\label{eq73}
0\lra\cO^{k_1}\lra E_1\lra F\lra0
\end{equation}
for some sheaf $F$.
\item[II] $d_1=n_1+g\ell<2n_1$, $\ell\ge1$. $E_1$ is the dual span of a stable bundle $E$ of slope $>2g$ and is therefore given by an exact sequence
\begin{equation}\label{eq55}
0\lra E^*\lra  H^0(E)^*\otimes\cO\lra E_1\lra0.
\end{equation}
Moreover, $h^0(E)=h^0(E_1)= n_1+\ell$.
\item[III] $d_1=n_1+g\ell+\ell'<2n_1$ with $\ell>0$ and $0<\ell'<g$. There exist exact sequences
\begin{equation}\label{eq56}
0\lra D_{E'}^*\lra H^0(E')\otimes\cO\lra E'\lra0
\end{equation}
with $E'$ stable of rank $n_1+\ell'$, degree $d_1$ and $h^0(E')=n_1+\ell+\ell'$, and
\begin{equation}\label{eq57}
0\lra \cO^{\ell'}\lra E'\lra E_1\lra 0.
\end{equation}
\end{itemize}
\end{theorem}

\begin{proof}
I. By \cite{bgn} (for $d_1\le n_1$) and \cite[Lemme 3-A-2]{m1} (for $n_1<d_1<n_1+g$), the sections of $E_1$ are all carried by a subsheaf $\cO^{k_1}$. This gives \eqref{eq73}

II. We know that $B(n_1,d_1,k_1)\ne\emptyset$ if and only if $d_1\ge n_1+g(k_1-n_1)$.  We can therefore choose $E_1\in B(n_1,d_1,n_1+\ell)$ and then $h^0(E_1)=n_1+\ell$. By \cite[Th\'eor\`eme 2-B-1]{m1}, $E_1=D_E$ for some stable bundle $E$ of rank $\ell$ and degree $d_1$, hence of slope $>2g$. So $h^0(E)=n_1+\ell$ and we have the exact sequence \eqref{eq55}.

III. By \cite[Th\'eor\`eme 2-B-1]{m1} (see Case II), there exists a stable bundle $E'$ of rank $n_1+\ell'$ and degree $d_1$ with $h^0(E')=n_1+\ell+\ell'$, fitting in an exact sequence \ref{eq56}. Now, as in the proof of \cite[Th\'eor\`eme 3-B-1]{m1}, we have an exact sequence \eqref{eq57}, with $E_1\in B(n_1,d_1,n_1+\ell)$.
Since $B(n_1,d_1,k_1)\ne\emptyset$, we have (see \cite{m1})
\[n_1+g\ell+\ell'=d_1\ge n_1+g(k_1-n_1),\]
 hence $\ell\ge k_1-n_1$ and $E_1\in B(n_1,d_1,k_1)$. 
\end{proof}

\begin{rem}\begin{em}\label{r23} Theorem \ref{t9} can be extended to the case $d_1=2n_1$ (see \cite{m2}). We have two cases.

\begin{itemize}
\item[IV] $C$ non-hyperelliptic. By the results of \cite{m2}, we are in one of Cases I-III, except that in Case II, we now have $\deg E=2n_1g$. There is one further case, namely when $E_1=D_{K_C}\in B(g-1,2g-2,g)$.
\item[V] $C$ hyperelliptic. In this case $B(n_1,2n_1,k_1)\ne\emptyset$ if and only if $k_1\le n_1$ (see \cite{m2}). When this holds, there exists $E_1\in B(n_1,2n_1,k_1)$ which admits  $\cO^{k_1}$ as a subsheaf.
\end{itemize}\end{em}\end{rem}

\section{A non-emptiness result for twisted BN loci}\label{twisted}
In this section we give an extended version of \cite[Theorem 1.1]{hhn}, which we need for comparison purposes.

\begin{theorem}\label{t1}
Let $C$ be a smooth curve of genus $g\ge2$ and $E_2$ any vector bundle of rank $n_2$ and degree $d_2$ on $C$. Let $d_0$ and $k_0$ be integers satisfying $\beta(1,d_0,k_0)(E_2)\ge1$, and suppose that $n_1\ge2$. Then,
\begin{itemize}
\item[(i)] if $k\le n_1k_0$  and $d_1\ge n_1d_0+1\  (\text{resp.}, d_1\ge n_1d_0)$, the twisted BN locus $B(n_1,d_1,k)(E_2)  \ (\text{resp.}, \widetilde{B}(n_1,d_1,k)(E_2))$ is non-empty and $\beta(n_1,d_1,k)(E_2)>(\text{resp.}, \ge)\,1$;
\item[(ii)] if $k_1:=k-n_2d_1+n_1d_2-n_1n_2(g-1)\le n_1k_0$ and $-d_1\ge n_1d_0+1\ (\text{resp.}, -d_1\ge n_1d_0)$, the twisted BN locus $B(n_1,d_1,k)(K_C\otimes E_2^*)\ (\text{resp.}, \widetilde{B}(n_1,d_1,k)(K_C\otimes E_2^*))$ is non-empty and 
\[\beta(n_1,d_1,k)(K_C\otimes E_2^*)>(\text{resp.}, \ge)\,1.\]
\end{itemize}
\end{theorem}

\begin{proof} (i) If $k\le0$, then $B(n_1,d_1,k)(E_2)=M_1$ and the result holda. Otherwise, the non-emptiness of $B(n_1,d_1,k)(E_2)$ (resp., $\widetilde{B}(n_1,d_1,k)(E_2)$) is  \cite[Theorem 1.1]{hhn}. The inequality $\beta(n_1,d_1,k)(E_2)>(\ge)1$ follows from the observation that
\begin{equation}\label{eq13}
\beta(n_1,n_1d_0,n_1k_0)(E_2)\ge1 \Longleftrightarrow \beta(1,d_0,k_0)(E_2)\ge1
\end{equation}
and the fact that $\beta(n_1,d_1,k)(E_2)$ is a strictly increasing function of $d_1$ and a decreasing function of $k$.

(ii) Under the stated hypotheses, in the case $-d_1\ge n_1d_0+1$, (i) implies that $B(n_1,-d_1,k_1)(E_2)\ne\emptyset$. By Serre duality, 
\[B(n_1,d_1,k)(K_C\otimes E_2^*)\cong B(n_1,-d_1,k_1)(E_2).\]
So $B(n_1,d_1,k)(K_C\otimes E_2^*)\ne\emptyset$. Moreover, as in the proof of (i), 
\begin{equation}\label{eq14}
\beta(n_1,-d_1,k_1)(E_2)>1
\end{equation} and
\[\beta(n_1,d_1,k)(K_C\otimes E_2^*)=n_1^2(g-1)+1-kk_1=\beta(n_1,-d_1,k_1)(E_2).\]

The same proof applies in the case $-d_1\ge n_1d_0$ if we replace  $>$ in \eqref{eq14} by $\ge$ and $B$ with $\widetilde{B}$ throughout.
\end{proof}

\begin{cor}\label{c7}
Under the hypotheses of Theorem \ref{t1}, the twisted BN locus $B^k(\cU_!,\cU_2)\ (\text{resp.}, \widetilde{B}(\cU_1,\cU_2))$ is non-empty. Moreover
\[\beta^k(\cU_1,\cU_2)>n_2^2(g-1)+2\  >0\ (\mbox{resp.,}\ge n_2^2(g-1)+2>0).\]
\end{cor}

\begin{proof}
In the theorem, we can take $E_2$ to be stable; the required non-emptiness then follows (see Remark \ref{r16}). The inequality for $\beta^k(\cU_1,\cU_2)$ follows from the fact that 
\[\beta^k(\cU_1,\cU_2)= \dim M_2+\beta(n_1,d_1,k)(E_2)\]
(see \eqref{eq5} and \eqref{eq2}).
\end{proof}

\begin{rem}\label{r3}\begin{em} (i)
As indicated in the introduction, the fact that the BN numbers are negative will be crucial for us in establishing that many of the examples constructed in this paper are new.

(ii) \eqref{eq13} makes Theorem \ref{t1}(i) the precise analogue of the theorem of Teixidor i Bigas \cite{te} for twisted BN loci $B(n_1,d_1,k)(E_2)$, at least on a general curve; the proof in \cite{hhn} depends on Mercat's construction \cite{m3} and holds for any curve. 

(iii) If we replace $E_2$ by $E_2\otimes L$, where $L$ is a line bundle of degree $\ell$, and simultaneously replace $d_0$ by $d_0-\ell$ and $d_1$ by $d_1-n_1\ell$ in Theorem \ref{t1}(i), the BN numbers and the isomorphism classes of the BN loci are unchanged. This applies also to Theorem \ref{t1}(ii) except that now $d_1$ must be replaced by $d_1+n_1\ell$.
\end{em}\end{rem}

\section{Examples with $k=k_1k_2$}\label{first}
In this section, we consider the following situation. Suppose that $B(n_i,d_i,k_i)\ne\emptyset$ for $i=1,2$ and that there exist $E_i\in B(n_i,d_i,k_i)$ such that $h^0(E_1\otimes E_2)\ge k_1k_2$. When $n_2=1$, this is exactly the situation considered in \cite{bmno}. Here, however, we assume $n_i\ge2$. One scenario to which this applies is when $E_1=D_E$.

\subsection{The construction}\label{ss41}
\begin{prop}\label{p3}
Let $C$ be a smooth curve of genus $g$. Suppose that $E$ is a generated stable (resp., semistable) bundle of rank $n$ and degree $d_1$ and that $h^0(E)=k_1=n+n_1$. Suppose further that $D_E$ is stable (resp., semistable) and that $E_2\in B(n_2,d_2,k_2)\ (\text{resp.}, [E_2]\in\widetilde{B}(n_2,d_2,k_2))$ with $\mu(E_2)<\mu(E)$. Then $h^0(D_E\otimes E_2)\ge k_1k_2\ (\text{resp.}, h^0(D_E\otimes\gr E_2)\ge k_1k_2)$. In particular, $B^k(\cU_1,\cU_2)\ne\emptyset\ (\text{resp.}, \widetilde{B}^k(\cU_1,\cU_2)\ne\emptyset)$.
\end{prop}
\begin{proof}
Consider the sequence \eqref{eq58} and suppose first that we are in the stable case. Note that $D_E\in B(n_1,d_1,k_1)$, since $D_E$ is stable by hypothesis and $h^0(E^*)=0$. Note also that $h^0(E^*\otimes E_2)=0$ since $E^*\otimes E_2$ is semistable of negative slope. Tensoring \eqref{eq58} by $E_2$ and taking global sections, we obtain
\[h^0(D_E\otimes E_2)\ge h^0(E)\cdot h^0(E_2)\ge k_1k_2.\]

The semistable case proceeds similarly.
\end{proof}

Note that, by \eqref{eq2}, we have
\begin{equation}\label{eq51}
\beta^k(\cU_1,\cU_2)=(n_1^2+n_2^2)(g-1)+2-k(k-\chi).
\end{equation}
We have also
\begin{equation}\label{eq52}
\beta(n_1n_2,n_2d_1+n_1d_2,k)=n_1^2n_2^2(g-1)+1-k(k-\chi).
\end{equation}

\begin{lem}\label{l10}Let $\mu_i$, $\lambda_i$ be positive rational numbers for $i=1,2$ and suppose that $d_i=n_i\mu_i$, $k_i=n_i\lambda_i$, $k=k_1k_2$. If
\begin{equation}\label{eq28}\mu_1+\mu_2<\lambda_1\lambda_2+g-1,
\end{equation}
then, for fixed $\mu_i$, $\lambda_i$ and sufficiently large $n_1$, $n_2$, $\beta^k(\cU_1,\cU_2)<0$.
\end{lem}
\begin{proof}
Since $k=k_1k_2$, \eqref{eq51} is equivalent to
\begin{equation}\label{eq62}
\frac{\beta^k(\cU_1,\cU_2)}{n_1^2n_2^2}=\left(\frac1{n_2^2}+\frac1{n_1^2}\right)(g-1)+\frac2{n_1^2n_2^2}-\lambda_1\lambda_2(\lambda_1\lambda_2-(\mu_1+\mu_2)+g-1).
\end{equation}
The result follows.
\end{proof}

\begin{theorem}\label{t6}
Suppose that $C$ is a smooth curve of genus $g\ge2$ and that $\mu_i$, $\lambda_i$ ($i=1,2$) are positive rational numbers, $\mu_1<2$ and $\mu_2\le2g$.
\begin{itemize}
\item[(i)] If $B(n_i,d_i,k_i)$ are non-empty BN loci with $n_i\ge2$ supported by $(\mu_i,\lambda_i)$ and $k=k_1k_2$, then $B^k(\cU_1,\cU_2)\ne\emptyset$.
\item[(ii)]  If $\widetilde{B}(n_i,d_i,k_i)$ are non-empty BN loci with $n_i\ge2$ supported by $(\mu_i,\lambda_i)$ and $k=k_1k_2$, then $\widetilde{B}^k(\cU_1,\cU_2)\ne\emptyset$ and
\[\widetilde{B}(n_1n_2,n_2d_1+n_1d_2,k)\ne\emptyset.\]
\end{itemize}
\end{theorem}
\begin{proof} (i) Let $E_2\in B(n_2,d_2,k_2)$ with $d_2\le2n_2g$. We prove that there exists $E_1\in B(n_1,d_1,k_1)$ such that $h^0(E_1\otimes E_2)\ge k_1k_2$. The proof is the same as in \cite{bmno} with $L$ replaced by $E_2$. We need to consider cases I-III of Theorem \ref{t9}.

I. $\cO^{k_1}\otimes E_2$ is a subsheaf of $E_1\otimes E_2$. It follows at once that $k_1k_2\le h^0(E_1\otimes E_2$).

II. in \eqref{eq55}, $E_1\cong D_E$ and $\mu(E_2)\le2g<\mu(E)$. The result follows from Proposition \ref{p3}.

III. Tensoring \eqref{eq56} and \eqref{eq57} by $E_2$ and noting that $D_{E'}^*\otimes E_2$ is semistable of negative slope, we obtain $h^0(E'\otimes E_2)\ge h^0(E') h^0(E_2)$ and hence
\[h^0(E_1\otimes E_2)\ge ( h^0(E')-\ell')h^0(E_2)=(n_1+\ell)h^0(E_2)\ge k_1k_2.\]

(ii) This follows in the same way as (i), together with the fact  that the tensor product of semistable bundles is semistable.
\end{proof}

\begin{rem}\label{r11}\begin{em}
If $C$ is non-hyperelliptic, the theorem holds also under the hypotheses $\mu_1\le2$, $\mu_2<2g$ (see Case IV in Remark \ref{r23}).
\end{em}\end{rem}

\subsection{Examples}\label{ss42}

Using Lemmas \ref{l2} and \ref{l5}, it is easy to find many cases as in Theorem \ref{t6} where \eqref{eq28} is also satisfied, thus yielding examples of non-empty BN loci $B^k(\cU_1,\cU_2)$ with negative BN number. These cannot be obtained from Theorem \ref{t1}.

\begin{ex}\label{ex2}\begin{em}
Suppose that $\mu_i\le\frac12$ for $i=1,2$. From \cite{bgn} (see Lemma \ref{l5} and Figure 1), we can take $\lambda_i$ to be any positive rational numbers with $\lambda_i\le1-\frac1g(1-\mu_i)$. Then, Lemma \ref{l10} and Theorem  \ref{t6} apply, giving $B^k(\cU_1,\cU_2)\ne\emptyset$ and $\beta^k(\cU_1,\cU_2)<0$ for sufficiently large $n_1$, $n_2$, even when $g=2$. More generally, we can take $\mu_1\le1$ and $\mu_2\le g-2$ for $g\ge3$, or $\mu_1<2$ and $\mu_2\le g-3$ for $g\ge4$, with $\lambda_i$ any positive rational numbers for which there exist non-empty BN loci $B(n_i,d_i,k_i)$ supported by $(\mu_i,\lambda_i)$. 
\end{em}\end{ex}

It is even possible to find examples with $n_1=n_2=2$.

\begin{ex}\label{ex1}\begin{em}
Suppose that $n_1=n_2=2$. Then, under the hypotheses of Theorem \ref{t6}, $\beta^k(\cU_1,\cU_2)<0$ if and only if
\begin{equation}\label{eq61}
\left(\frac12-\lambda_1\lambda_2\right)(g-1)+\frac18-\lambda_1\lambda_2(\lambda_1\lambda_2-(\mu_1+\mu_2))<0
\end{equation}
(see \eqref{eq62}). We can take $B(n_i,d_i,k_i)=B(2,3,2)$ for $i=1,2$; this is non-empty by either Lemma \ref{l2} or Lemma \ref{l5}. \eqref{eq61} now becomes
\[-\frac12(g-1)+\frac18+2<0,\]
which is true for $g\ge6$. Here, $B(2,3,2)$ is the only candidate for $B(2,d_1,k_1)$ for applying Theorem \ref{t6}, but, for $i=2$, we could take any non-empty $B(2,d_2,k_2)$ instead of $B(2,3,2)$; once we have fixed $d_2$ with $d_2\le4g$ and $k_2\ge2$, \eqref{eq61} will hold for sufficiently large $g$. 
\end{em}\end{ex}

\subsection{A new region in the BN map}\label{ss43}
Substituting $k=k_1k_2$ in \eqref{eq52} and dividing by $n_1^2n_2^2$, we have
\begin{multline}\label{eq29}
\beta:=\frac{\beta(n_1n_2,n_2d_1+n_1d_2,k)}{n_1^2n_2^2}=\\
=g-1+\frac1{n_1^2n_2^2}-\lambda_1\lambda_2(\lambda_1\lambda_2-(\mu_1+\mu_2)+g-1).
\end{multline}
It is an interesting question as to whether $\beta$ can ever be negative under the hypotheses of Theorem \ref{t6} and the assumption that $\beta(n_2,d_2,k_2)\ge1$. One can see easily that $\beta>0$ if $\lambda_1\le1$, but other cases are not so straightforward. Independently of this, one can ask whether the theorem can give rise to new examples of non-empty BN loci in $M(n_1n_2,n_2d_1+n_1d_2,k_1k_2)$. Our next result shows that this is indeed the case.

We define a new region of the BN map as follows. Let
\[R:=\left\{(\mu_1,\mu_2,\lambda_1,\lambda_2)\in{\mathbb Q}^4\,\left|\begin{array}{c}\,0<\mu_1<2,0<\mu_2\le2g-2-\mu_1,\\\lambda_i\le\max\{ f_g(\mu_i),t_g(\mu_i)\}\end{array}\right.\right\}\]
and let $h:R\to {\mathbb Q}^2$ be defined by 
\[h(\mu_1,\mu_2,\lambda_1,\lambda_2)=(\mu_1+\mu_2,\lambda_1\lambda_2).\]

\begin{df}\label{d1}\begin{em}
The region BPN in the BN map is defined to be the union of $\im h$ and its Serre dual.
\end{em}\end{df}

\begin{theorem}\label{t8}
Let $C$ be a smooth  curve of genus $g\ge2$ and $n_i\ge2$. 
\begin{itemize}
\item[(i)] If $(\mu,\lambda)=h(\mu_1,\mu_2,\lambda_1,\lambda_2)\in\operatorname{BPN}$, then $(\mu_1+\mu_2,\lambda_1\lambda_2)$ supports a BN locus $\widetilde{B}(n_1n_2,n_2d_1+n_1d_2,k)$ with $k=k_1k_2$ for infinitely many values of $n_1$, $n_2$;
\item[(ii)] If $g\ge5$, $\operatorname{BPN} $ is not contained in $\operatorname{T\cup BMNO}$.
\end{itemize}
\end{theorem}

\begin{proof}
(i) By Lemmas \ref{l2} and \ref{l5}, $(\mu_i,\lambda_i)$ supports the BN locus $\widetilde{B}(n_i,d_i,k_i)$ ($i=1,2$) for infinitely many values of $n_i$. The result follows from Theorem \ref{t6}(ii).

(ii) Suppose that $2<\mu<4$. We can write $\mu=\mu_1+\mu_2$ with $1<\mu_i<2$ for $i=1,2$.  Now, if $\lambda_i=1+\frac1g(\mu_i-1)$, then, by \cite{m1} (see also Lemma \ref{l5}), $(\mu_i,\lambda_i$) is in the region BMNO and $(\mu_i,\lambda_i)$ supports $B(n_i,d_i,k_i)$ for any allowable $n_i$. So, by Theorem \ref{t6}(ii), $\widetilde{B}(n_1n_2,n_2d_1+n_1d_2,k)\ne\emptyset$ when $k=k_1k_2$. On the other hand,
\begin{equation}\label{eq53}
\lambda_1\lambda_2>1+\frac1g(\mu_1+\mu_2-2)
\ge1+\frac1g(\mu_1+\mu_2-\left\lceil\mu_1+\mu_2\right\rceil+1).
\end{equation}
It follows from Lemma \ref{l5} and \eqref{eq27} that $(\mu_1+\mu_2,\lambda_1\lambda_2)$ is not in the region BMNO if $\widehat{\eta}(2)\ge5$, in other words, if $g\ge7$. It is clear also that $(\mu_1+\mu_2,\lambda_1\lambda_2)$ is not in the region T, whose upper boundary is given by $\lambda=1$ in this range. In fact, our construction also gives new points for $g=5$ and $g=6$.
\end{proof}

\begin{rem}\label{r14}\begin{em}
 The construction in the proof of Theorem \ref{t8} shows that, for $g\ge7$, the top boundary of BPN in the range $2<\mu<4$ is given by $\lambda=1+\frac1g(\mu-2)+\frac1{g^2}\left(\frac{\mu-2}2\right)^2$, whereas the top boundary of BMNO is given by $\lambda=1+\frac1g(\mu-2)$ for $2<\mu<3$ and by $\lambda=1+\frac1g(\mu-3)$ for $3<\mu<4$. At the point $(2,1)$, the line $\lambda=1+\frac1g(\mu-2)$ is the tangent to the parabola $\lambda=1+\frac1g(\mu-2)+\frac1{g^2}\left(\frac{\mu-2}2\right)^2$. For $g=10$, BPN is illustrated in Figure 2, which extends to the range $4<\mu<5$. Figure 2 also shows an expanded version of the BN map close to $\mu=3$.
 \end{em}\end{rem}

\begin{center}
\includegraphics[width=12.5cm]{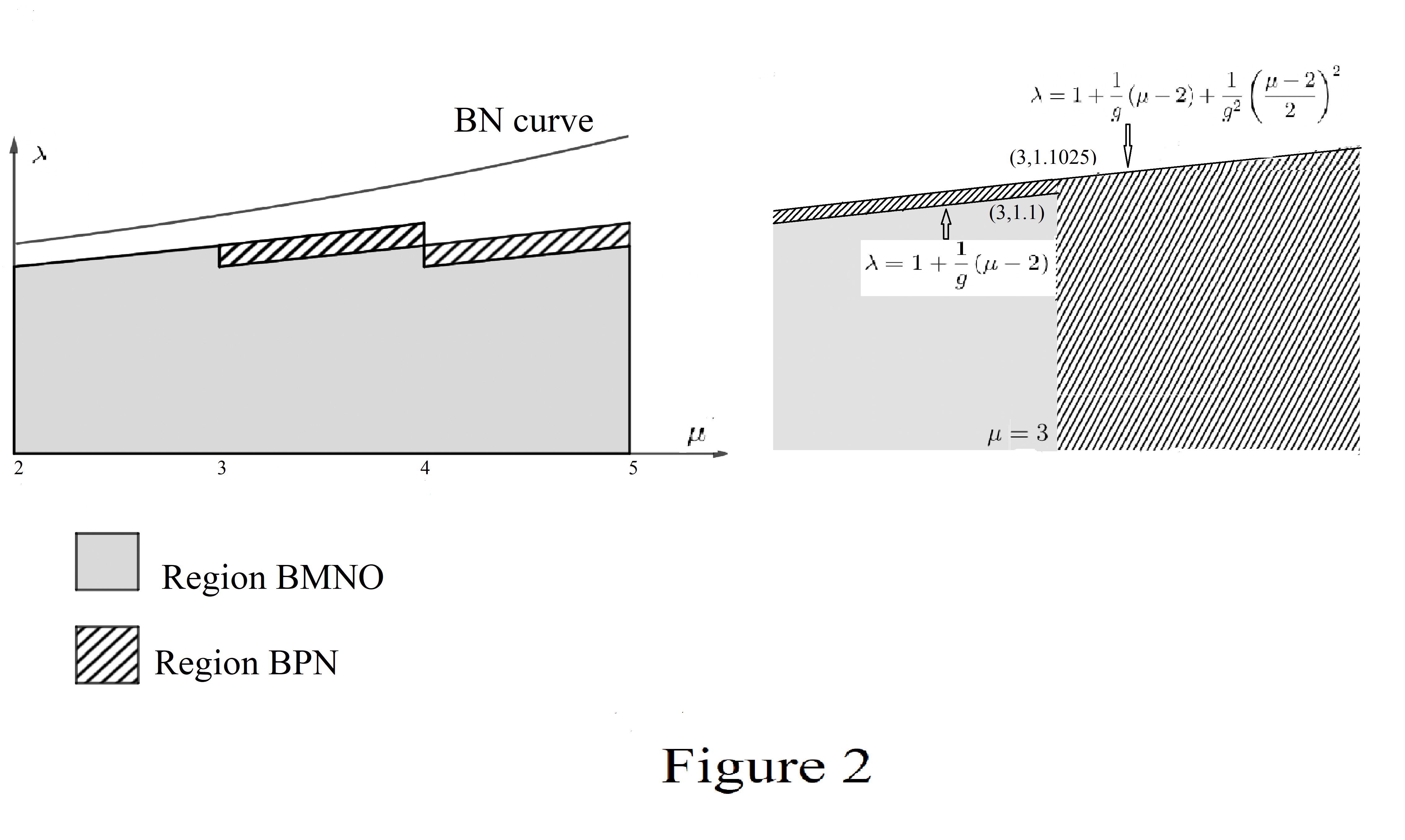}
\end{center}

\begin{rem}\begin{em}\label{r26}(i) BPN is a new region in the BN map (see Figure 2), whereas the papers \cite{ln2,ln5} give examples for isolated values of $\mu$, some of which extend to higher values of $g$. When $g=4$, on the other hand, there is an additional grey area of the BN map (see \cite[section 7]{ln1}), which contains the new part of the region BPN in the range $2<\mu<3$.

(ii) At least for Petri curves, some of the non-empty BN loci  
\[\widetilde{B}(n_1n_2,n_2d_1+n_1d_2,k)\subset \widetilde{M}(n_1n_2,n_2d_1+n_1d_2)\] 
constructed in Theorem \ref{t8} are new.  It is possible to choose the ranks and degrees so that $\gcd(n_1n_2,n_2d_1+n_1d_2)=1$; for example, if $g$ is odd, one can take
\[(n_1,d_1,k_1)=(g+1,2g+1,g+2),\ (n_2,d_2,k_2)=(g+2,2g+2,g+3).\]
Then $B(n_1n_2,n_2d_1+n_1d_2,k)\ne\emptyset$.

(iii) One can check that, in the framework of Theorem \ref{t8}, the examples of twisted BN loci with negative expected dimension constructed in \cite[Section 9]{hhn} give rise to points of T. It follows that our examples are different.
\end{em}\end{rem}

\section{Construction using kernel bundles}\label{kernel}
\subsection{The construction}\label{ss51}
The essential ingredient in proving Theorem \ref{t1} is Mercat's construction \cite{m3}, while Theorem \ref{t6} depends on the constructions of \cite{bgn, m1}. Another way of constructing non-empty twisted BN loci is described in \cite{br} and uses kernel bundles. The following proposition is contained in \cite[Remark 5.5]{br} and plays an essential r\^{o}le in the proof of \cite[Theorems 5.6 and 5.7]{br}. We include a proof for the convenience of the reader.
\begin{prop}\label{t2} 
Suppose that $E_1$ has rank $n_1$ and degree $d_1$ with $h^0(E_1)\ge k_1$ and that $E$ is as in \eqref{eq6}.
If $h^1(E_1\otimes E)=0$ and
\begin{equation}\label{eq12}
k\le (d-n(g-1))(k_1-n_1)-nd_1,
\end{equation}
then $h^0(E_1\otimes D_E^*)\ge k$.
\end{prop}
\begin{proof} We have $h^0(E_1\otimes E)=nd_1+n_1d-nn_1(g-1)$ and $h^0(E)\ge d-n(g-1)$. So, tensoring \eqref{eq6} by $E_1$ and taking global sections,
\[h^0(E_1\otimes D_E^*)\ge (d-n(g-1))k_1-nd_1-n_1d+nn_1(g-1).\]
Simplifying and comparing with \eqref{eq12}, this gives the result.
\end{proof}

We have the following immediate corollary (see Remark \ref{r16}).
\begin{cor}\label{c1}
Suppose that the hypotheses of Proposition \ref{t2} are satisfied.
\begin{itemize}
\item[(i)] If $E_1$ is stable $(\text{resp., semistable})$, then  $B(n_1,d_1,k)(D_E^*)\ne\emptyset$ $(\text{resp.}, \widetilde{B}(n_1,d_1,k)(D_E^*)\ne\emptyset)$.
\item[(ii)] If $D_E^*$ is stable $(\text{resp., semistable})$, then $B(h^0(E)-n,-d,k)(E_1)\ne\emptyset\ (\text{resp.},\widetilde{B}(h^0(E)-n,-d,k)(E_1)\ne\emptyset)$.
\item[(iii)] If both $E_1$ and $D_E^*$ are stable $(\text{resp., semistable})$ and $(n_2,d_2)=(h^0(E)-n,-d)$, then  $B^k(\cU_1,\cU_2)\ne\emptyset\ (\text{resp.}, \widetilde{B}^k(\cU_1,\cU_2)\ne\emptyset)$.
\end{itemize}\end{cor}

Our object here is to use Corollary \ref{c1}(iii) to obtain new non-empty examples of twisted BN loci. We assume that $n_1\ge2$. In order for \eqref{eq12} to yield positive values of $k$ when $d>n(g-1)$, which will always be the case, we clearly require $k_1>n_1$. 

 The key point we require in order to apply Corollary \ref{c1}(iii) is that $D_E^*$ and $E_1$ are stable. We have already discussed the existence of stable $E_1$ in subsection \ref{ss22} and the stability of $D_E$ in subsection \ref{ss25}. 

We can now state and prove a general theorem, which includes Theorem \ref{t5}.

\begin{theorem}\label{t4} Suppose that  $C$ is a smooth curve of genus $g\ge 2$, $n_1\ge2$, $k_1>n_1$ and either $d>2ng$ or $d=2ng$ and $C$ is non-hyperelliptic. Suppose further that $B(n_1,d_1,k_1)\ne\emptyset$ and that \eqref{eq12} holds. Let  $(n_2,d_2)=(d-ng,-d)$ and suppose that $S$ is defined as in \eqref{eq43}. Then $S\ne\emptyset$ and the morphism
\begin{equation}\label{eq22}
B(n_1,d_1,k_1)\times S\lra M_1\times M_2:(E_1,E)\mapsto (E_1,D_E^*)
\end{equation}
is injective and has image contained in $B^k(\cU_1,\cU_2)$; in particular,
\[B^k(\cU_1,\cU_2)\ne\emptyset.\]
If, in addition, $k=d(k_1-n_1)-e$, where
\begin{equation}\label{eq70}
e\ge n(g-1)(k_1-n_1)+nd_1,
\end{equation}
 and
\begin{equation}\label{eq34}
 d_1<k_1+n_1(g-1)-\frac{g-1}{k_1-n_1},
\end{equation}
then,  for any fixed values of $n_1$, $d_1$, $k_1$, $n$ and $e$ satisfying \eqref{eq70}, and $d\gg0$,
\begin{equation}\label{eq25}\beta^k(\cU_1,\cU_2)<0.
\end{equation}
\end{theorem}
\begin{proof}  The first part follows from  Theorem \ref{l1}, Proposition \ref{t2} and Corollary \ref{c1}(iii).
We need to note that $h^1(E_1\otimes E)=0$ since $E_1\otimes E$ is semistable of slope greater than $2g$.

For the second part, we have
\begin{eqnarray}\label{eq15}
\nonumber\beta^k(\cU_1,\cU_2)&=&n_1^2(g-1)+1+(d-ng)^2(g-1)+1\\&&-(d(k_1-n_1)-e)(d(k_1-n_1)-e-\chi),
\end{eqnarray}
where, by \eqref{eq24},
 \begin{eqnarray}\label{eq47}
\nonumber\chi&=&(d-ng)d_1-n_1d-n_1(d-ng)(g-1)\\&=&(d-ng)(d_1-n_1g)-nn_1g.
\end{eqnarray}
The right-hand side of \eqref{eq15} is a quadratic in $d$, with leading coefficient 
\begin{equation}\label{eq77}
g-1-(k_1-n_1)(k_1-n_1-d_1+n_1g).
\end{equation}
 So $\beta^k(\cU_1,\cU_2)<0$ for $d\gg0$ provided that \[d_1<k_1+n_1(g-1)-\frac{g-1}{k_1-n_1}.\]
\end{proof}

\begin{rem}\label{r4}\begin{em}
(i) Under the hypotheses of Theorem \ref{t4}, including the inequalities \eqref{eq70} and \eqref{eq34}, we deduce that the non-emptiness of $B^k(\cU_1,\cU_2)$ cannot be obtained from Theorem\  \ref{t1}. In fact, there are several choices we can make in attempting to apply Theorem \ref{t1} to this situation. We can take $E_2$ to be any of $E_1\otimes L$, $D_E^*\otimes L$, $K_C\otimes E_1^*\otimes L$ and $K_C\otimes D_E\otimes L$, where $L$ is any line bundle, with the appropriate choice of $(n_1,d_1)$ in each case. In all cases, we obtain $\beta^k(\cU_1,\cU_2)\ge1$. In the absence of \eqref{eq25}, we would have to work through all these cases separately.

(ii) If we assume only that $\widetilde{B}(n_1,d_1,k_1)\ne\emptyset$ and $d\ge2ng$ with $C$ any smooth curve of genus $g\ge2$, we obtain $\widetilde{B}(\cU_1,\cU_2)\ne\emptyset$. Also, if \eqref{eq34} holds and $d\gg0$, then the non-emptiness of $\widetilde{B}(\cU_1,\cU_2)$ cannot be obtained from  Theorem \ref{t1}.
\end{em}\end{rem}

\subsection{Examples}\label{ss52}
To show that Theorem \ref{t4} provides the examples we are seeking, we need to show that the conditions $B(n_1,d_1,k_1)\ne\emptyset$ and \eqref{eq34} are compatible. For this, we need more precise statements concerning the possible values of $d_1$.

\begin{cor}\label{c6}
Suppose that  $C$ is a smooth curve of genus $g\ge 3$, $n_1\ge2$, $k_1>n_1$ and $(n_2,d_2)=(d-ng,-d)$ with either $d>2ng$ or $d=2ng$ and $C$ non-hyperelliptic. Suppose further  that
\begin{equation}\label{eq63}
k_1+n_1(g-1)-n_1\left\lfloor\frac{g-1}{\lceil k_1/n_1\rceil}\right\rfloor\le d_1<k_1+n_1(g-1)-\frac{g-1}{k_1-n_1}
\end{equation}
and that $d_1$ is not divisible by $n_1$. Then, $B(n_1,d_1,k_1)\ne\emptyset$. Moreover, for
any fixed value of $e$ satisfying \eqref{eq70} and $d\gg0$, $B^k(\cU_1,\cU_2)\ne\emptyset$, but $\beta^k(\cU_1,\cU_2)<0$. In particular, $B^k(\cU_1,\cU_2)$ is of negative expected dimension and the non-emptiness of $B^k(\cU_1,\cU_2)$ cannot be obtained from Theorem \ref{t1}.
\end{cor}

\begin{proof}The fact that $B(n_1,d_1,k_1)\ne\emptyset$ follows from Remark \ref{r21}. Theorem \ref{t4} then implies that $B^k(\cU_1,\cU_2)\ne\emptyset$ and that $\beta^k(\cU_1,\cU_2)<0$ when \eqref{eq70} holds and $d\gg0$. Hence, the non-emptiness cannot be obtained from Theorem \ref{t1} (see Remark \ref{r3}(i)).
\end{proof}

To obtain examples from this corollary, we need, in particular, to prove the existence of $d_1$ satisfying \eqref{eq63}. This is equivalent to showing that
\begin{equation}\label{eq68}
n_1\left\lfloor\frac{g-1}{\lceil k_1/n_1\rceil}\right\rfloor>\frac{g-1}{k_1-n_1}
\end{equation}
In fact \eqref{eq68} cannot hold if $g=2$ (which is why we have assumed $g\ge3$ in Corollary \ref{c6}) or, more generally, if $k_1\ge n_1(g-1)+1$, since, in that case, the left-hand side of \eqref{eq68} is $0$.
We therefore assume that
\begin{equation}\label{eq69}
n_1<k_1\le n_1(g-1)
\end{equation}

\begin{ex}\label{ex5}\begin{em}
Suppose that $g=3$ and \eqref{eq69} holds. Then \eqref{eq68} holds if and only $n_1(k_1-n_1)>2$, in other words, either $n_1\ge3$ or $n_1=2$ and $k_1=4$. In the latter case, the only possible value for $d_1$ is $d_1=6$. This is not permitted by the hypotheses of Corollary \ref{c6} and, in fact $B(2,6,4)=\emptyset$. For $n_1\ge3$, there do exist solutions of \eqref{eq63} for which $d_1$ is not divisible by $n_1$. In fact, if $k_1=n_1+2$, \eqref{eq63} gives $2n_1+2\le d_1<3n_1+1$, giving $n_1-2$ values of $d_1$ not divisible by $n_1$. If $n_1+3\le k_1\le2n_1$, there are $n_1-1$ permitted values of $d_1$ in the range $k_1+n_1\le d_1<k_1+2n_1$. A similar analysis can be carried out for $g\ge4$; in this case, in particular, $B(2,4g-5,2g-2)\ne\emptyset$ and \eqref{eq63} is satisfied, so examples exist for $n_1=2$, as well as for larger $n_1$.
\end{em}\end{ex}

\begin{rem}\label{r25}\begin{em}When $k_1=n_1+1$, there are additional cases in which the non-emptiness of $B(n_1,d_1,k_1)$ is known. This is related to the extended dual span construction, which involves the stability of bundles $D_{L,V}$, where $L$ is a line bundle and $V$ is a subspace of $H^0(L)$ which generates $L$. Applications of this will be considered in a paper currently in preparation \cite{bpn}.
\end{em}\end{rem}

\end{document}